\documentclass[11pt,a4paper]{article}
\usepackage{indentfirst}
\usepackage{amsfonts}
\usepackage{amssymb}
\usepackage{mathrsfs}
\usepackage{amsmath}
\usepackage{amsthm}
\usepackage{enumerate}
\usepackage{cite}
\usepackage{mathrsfs}
\allowdisplaybreaks
\usepackage{geometry}
\usepackage{ifpdf}
\ifpdf
  \usepackage[colorlinks=true,linkcolor=blue,citecolor=red, final,backref=page,hyperindex]{hyperref}
\else
  \usepackage[colorlinks,final,backref=page,hyperindex,hypertex]{hyperref}
\fi

\usepackage{indentfirst}
\usepackage{latexsym}
\usepackage[all]{xy}

\def\pr{\partial}
\def\<{\langle}
\def\>{\rangle}

\def\c{\cdot}

\def\l{\lambda}

\def\ci{\circ}

\newcommand{\bmx}{\begin{pmatrix}}
\newcommand{\emx}{\end{pmatrix}}

\newtheorem{thm}{Theorem}[section]
\newtheorem{lem}[thm]{Lemma}
\newtheorem{cor}[thm]{Corollary}
\newtheorem{pro}[thm]{Proposition}
\newtheorem{ex}[thm]{Example}

\theoremstyle{definition}
\newtheorem{defi}{Definition}[section]

\theoremstyle{remark}
\newtheorem{rmk}{Remark}[section]
\begin{document}
\title{\bf Jacobi-Jordan conformal algebras: Basics, Constructions and related structures}
\author{T. Chtioui, S. Mabrouk, A. Makhlouf}
\author{{ Taoufik Chtioui$^{1}$
 \footnote{ E-mail: chtioui.taoufik@yahoo.fr}
, Sami Mabrouk$^{2}$
 \footnote{  E-mail: mabrouksami00@yahoo.fr}
, Abdenacer Makhlouf$^{3}$
 \footnote{  E-mail: Abdenacer.Makhlouf@uha.fr (Corresponding author)}
}\\
{\small 1.  University of Gabes, Faculty of Sciences Gabes,
 City Riadh 6072 Zrig, Gabes, Tunisia.} \\
{\small 2.  University of Gafsa, Faculty of Sciences, 2112 Gafsa, Tunisia}\\
{\small 3.~ Université de Haute Alsace, IRIMAS - Département de Mathématiques,} \\ {\small 18, rue des frères Lumière,
F-68093 Mulhouse, France}}
 \maketitle

\begin{abstract}
The main purpose of  this paper is to introduce and investigate the notion of Jacobi-Jordan conformal  algebra. They are a generalization of  Jacobi-Jordan algebras which correspond to the case in which the formal parameter  $\l$ equals $0$. We consider some related structures such as conformal modules, corresponding representations and $\mathcal{O}$-operators. Therefore,  conformal derivations from  Jacobi-Jordan conformal algebras to their conformal
modules are used to describe conformal derivations of Jacobi-Jordan conformal
algebras of semidirect product type. 
Moreover, we study a class of Jacobi-Jordan conformal algebras called quadratic Jacobi-Jordan conformal algebras, which are characterized by mock-Gel'fand Dorfman bialgebras.
Finally, the  $\mathbb{C}[\pr]$-split extending structures problem for Jacobi-Jordan conformal algebras is studied. Furthermore,  we introduce an  unified product of a given Jacobi-Jordan conformal
algebra $J$ and a given  $\mathbb{C}[\pr]$-module $K$. This product includes some other interesting products of Jacobi-Jordan conformal algebras such as twisted product or crossed product.
Using this product, a cohomological type object is constructed to provide a theoretical answer to the  $\mathbb{C}[\pr]$-split extending structures problem.

\end{abstract}
\textbf{Keywords}: Jacobi-Jordan conformal algebra,  conformal module,  mock-Gel'fand-Dorfman bialgebra, unified product, extending structure.\\
\textbf{Mathematics Subject Classification}  (2020): 17A30, 17A60, 17B69. 
\numberwithin{equation}{section}

\tableofcontents

\section{Introduction}

In the jungle of non-associative algebras, {\bf Jacobi-Jordan algebras} 
are rather special objects. A Jacobi-Jordan algebra is a vector space $J$
together with a commutative bilinear map $\bullet : J \times J \to J$ such
that   $$a\bullet (b\bullet c)  + b\bullet (c\bullet a)  + c\bullet (a\bullet b)  = 0$$ for all $a$, $b$, $c\in J$.
According to \cite{zu2}, where a detailed history of Jacobi-Jordan algebras is given and
several conjectures are proposed, this family
of algebras was first defined in \cite{Zh} and since then they have been studied independently in various papers \cite{BB0, BB,BB1,BB2, KarimaTaoufikAtefSami, BF, Camacho,
GK, Ok, wa, Buse, Zh2} under different names such as Lie-Jordan
algebras, Jordan algebras of nilindex $3$, pathological
algebras, mock-Lie algebras or Jacobi-Jordan algebras. Throughout this paper, we shall adopt the name
Jacobi-Jordan algebra.

An algebraic formalization of the properties of the operator product expansion (OPE) in two-dimensional conformal field theory \cite{BPZ} gave rise to a new class of algebraic systems, vertex operator algebras \cite{BO,F2}. The notion of a Lie conformal algebra encodes the singular part of the
OPE which is responsible for the commutator of two
chiral fields \cite{KAC}. Roughly speaking, Lie conformal algebras correspond to vertex algebras
by the same way as Lie algebras correspond to their associative enveloping algebras.


The structure theory of finite (i.e., finitely
generated as $\mathbb C[\partial]$-modules) associative and Lie conformal algebras was developed in \cite{DK} and
later generalized in \cite{BDK} for pseudoalgebras over a wide class of cocommutative
Hopf algebras. From the algebraic point of view, the notions of conformal algebra \cite{DK}, their representations \cite{CK}
and cohomologies \cite{BKV,D,KK} are higher-level analogues of the ordinary notions in
the pseudo-tensor category \cite{BD2} associated with the polynomial Hopf algebra
(see \cite{BDK} for a detailed explanation).

Some features of the structure theory of conformal algebras (and their representations)
of infinite type were also considered in a series of works \cite{BKL1,BKL2,DK2,KAC2,Re1,Re2,Z1,Z2}. In this field,
one of the most relevant problem is to describe the structure of conformal
algebras with faithful irreducible representation of finite type (these algebras could be
of infinite type themselves). In \cite{BKL1,KAC2},  conjectures on the structure of such algebras
(associative and Lie) were stated. They were proved  under some additional conditions in \cite{BKL1,DK,Z2}. Another relevant problem is to classify simple and semisimple conformal algebras of linear
growth (i.e., of Gel'fand-Kirillov dimension one).
This problem was solved for finitely generated associative conformal algebras which
contain a unit \cite{Re1,Re2}, or at least an idempotent \cite{Z1,Z2}. A structure theory of associative conformal algebras with finite faithful representation similar to those examples of conformal algebras stated in these papers was developed in \cite{KO}. 

The  split extending structures problem 
for some classical algebra objects such as groups, associative algebras, Hopf algebras,
Lie algebras, Leibniz algebras and left-symmetric algebras have been studied in \cite{am-2011, am0-2014, am-2014, am-2015, am-2019, H2}
respectively. In fact, this problem generalizes two important algebra problems. For associative and  Lie
conformal algebras, the first one is the $\mathbb C[\pr]$-split extension problem   introduced by Y.Y. Hong in \cite{H1, H3}, while the second one is the $C[\partial]$-split factorization problem.


The main purpose of  this paper is to introduce and investigate the notion of Jacobi-Jordan conformal  algebra. It
 is organized as follows. In Section \ref{JJConformal}, we introduce a definition  of Jacobi-Jordan conformal algebra and give some basic results. We study $\mathcal O$-operators on Jacobi-Jordan algebras and  symplectic Jacobi-Jordan algebras. In Section \ref{Characterization}, we provide  a  characterization of quadratic Jacobi-Jordan conformal algebras through Mock-Gel’fand Dorfman bialgebras and
present various constructions. The  $\mathbb{C}[\pr]$-split extending structures problem for Jacobi-Jordan conformal algebras is studied in Section \ref{ExtendingStructures}. We introduce a definition of unified product of a given a Jacobi-Jordan conformal
algebra and consider some interesting sub-classes like twisted and crossed products of Jacobi-Jordan conformal algebras.

Throughout the paper, all algebraic systems are supposed to be over a field $\mathbb{F}$
of characteristic $0$.  We denote by $\mathbb{Z}_{+}$ the set of all nonnegative integers and by $\mathbb{Z}$ the set of all integers.
\section{Jacobi-Jordan  conformal algebras and their modules}\label{JJConformal}

In this section, we provide some basic notions of (left) anti-associative and Jacobi-Jordan conformal algebras along with their conformal modules.   We introduce the notion of $\mathcal O$-operator on Jacobi-Jordan conformal algebras with respect a given module as a generalization of Rota-Baxter operators. 

\subsection{Definitions and Basic results}

A conformal algebra is a vector space $A$ equipped with a linear map $\pr: A \to A$ and a countable family of bilinear operations $-_{(n)}-: A\times A \to A, n \in \mathbb{Z}_+$, satisfying the following axioms:
\begin{itemize}
  \item [(i)] For every $a,b \in A$, there exists $N \in \mathbb{Z}_+$ such that $a_{(n)}b=0$  for all $n \geq N$,
  \item [(ii)] $\pr a_{(n)} b=-na_{(n-1)} b$,
  \item [(iii)] $a_{(n)} \pr b=\pr(a_{(n)} b)+ na_{(n-1)} b$.
\end{itemize}
Every conformal algebra $A$ is a left module over the polynomial algebra $\mathbb{F}[\pr]$.
The structure of a conformal algebra on an $\mathbb{F}[\pr]$-module $A$ may be expressed by means of a single polynomial-valued map called $\l$-product
$$-_\l -: A \times A \to A[\l],\quad a_\l b=\sum_{n \in \mathbb{Z}_+} \l^{(n)} a_{(n)} b,$$
where $\l $ is a formal variable and  $\l^{(n)}=\l^n/n!$.  In terms of the $\l$-product, the definition of a conformal algebra is as follows.
\begin{defi}
A conformal algebra  is a module $A$ over the polynomial algebra  $\mathbb{F}[\pr]$
equipped with a polynomial-valued $\mathbb{F}$-linear operation
$A \times A \to A[\lambda]$ denoted by $a \times b \to a_\lambda b$,
where $\l$  is a formal variable (this operation is called $\l$-product), satisfying the
following axioms
\begin{align}\label{conf alg}
   & \partial a_\lambda b=-\lambda a_\lambda b,\ \ a_\lambda \partial b=(\lambda+\partial)a_\lambda b, \hspace{2 cm} (\textrm{Conformal sesquilinearity}) .\end{align}
\end{defi}
A conformal algebra is said to be finite if it is finitely generated as a module
over $\mathbb{F}[\pr]$.
The rank of a finite conformal algebra $A$ is its rank as a  $\mathbb{F}[\pr]$-module.

Let $(A,\cdot_\l)$ be a conformal algebra.  Then by conformal sesquilinearity, the following equalities  hold true
\begin{align}\label{lemma}
    (a_{-\l -\pr} b)_{\l+\mu}c=&(a_\mu b)_{\l+\mu} c,\\
   a_\mu  (b_{-\l -\pr} c)=& a_\mu  (b_{-\l -\mu-\pr} c),\\
   (a_{-\mu-\pr } b)_{-\l-\pr} c=& (a_{-\l-\mu-\pr}b)_{-\l-\pr} c
\end{align}
for all $a,b ,c \in A$. \\
We prove the first identity by the following computation.
\begin{align*}
 (a_{-\lambda-\partial}b )_{\lambda+\mu} c& =\bigg(\sum_{k=0}^n \frac{(-\l-\pr)^k}{k!}
a_{(k)}b \bigg)_{\l+\mu} c \\
&=\sum_{k=0}^n \frac{(-\l+\l+\mu)^k}{k!}((a_{(k)}b)_{\l+\mu}c) \quad \textrm{(by conformal sesquilinearity) }\\
&=\bigg(\sum_{k=0}^n\frac{\mu^k}{k!}a_{(k)}b \bigg)_{\l+\mu}c
=(a_{\mu} b)_{\lambda+\mu}c.
\end{align*}

Let $M, N$ be $\mathbb{F}[\pr]$-modules.  A conformal linear map from $M$ to $N$ is a sequence $f=\{f_{(n)}\}_{n \geqslant 0}$ of $f_{(n)} \in Hom_\mathbb{F}(M,N)$ satisfying 
\begin{align*}
    \pr  f_{(n)} -f_{(n)} \pr =-n f_{(n-1)},\quad \forall n \in \mathbb{N}.
\end{align*}
In particular, $f_{(0)}$ is an $\mathbb{F}[\pr]$-module homomorphism.  Set $f_\l= \displaystyle \sum_{k=0}^{\infty} \frac{\l^n}{n!} f_{(n)}$. Then   $f=\{f_{(n)}\}_{n \geqslant 0}$ is a conformal linear map if and only if $f_\l :M \to N[\l]$  is $\mathbb{F}$-linear and $f_\l \pr=(\l+\pr)f_\l$.

We denote by  $Chom(M,N) $  the set of conformal linear maps from $ M$ to $ N$. It turns out that 
$Chom(M,N)$ is a $\mathbb{F}[\pr]$-module via:
$$\pr f_{(n)}=-nf_{(n-1)}, \quad \text{equivalently},\quad \pr f_\l =-\l f_\l . $$
The composition $fg : L \to  N$ of conformal linear maps $f : M \to  N$  and $g : L \to  M$
is given by
\begin{align*}
    (fg)_{(n)}=\sum_{k=0}^n \begin{pmatrix}
                              n \\
                              k \\
                            \end{pmatrix}
        f_{(k)}g_{(n-k)},   \quad \text{equivalently},\quad (f_\l g)_{\l+\mu}=f_\l g_\mu.
\end{align*}
If $M$  is a finitely generated $\mathbb{F}[\pr]$-module, then $Cend(M) := Chom(M,M)$  is an
associative conformal algebra with respect to the above composition. Thus, $Cend(M)$ becomes a Lie conformal algebra, denoted as $gc(M)$,
with respect to the following $\l$-bracket:
\begin{align*}
[f_\l g]_\mu=f_\l g_{\mu-\l}-g_{\mu-\l}f_\l,   \quad \text{equivalently},\quad [f_\l g]=f_\l g-g_{-\l-\pr}f.
\end{align*}
Define the conformal dual of an $\mathbb{F}[\pr]$-module $M$ as $M^{*c}=Chom(M,\mathbb{F})$, that is 
$$M^{*c}=\{f: M \to \mathbb{F}[\l]\; | \; f\ \textrm{is }\  \mathbb{F}\textrm{-linear and}\ f_\l(\pr m)=\l f_\l(m)\} .     $$        

\begin{lem}\label{Lemme one}
Let $f,g \in Cend(M)$, then for any $m \in M$ we have 
\begin{itemize}
    \item [(1)] $f_\l (g_{-\mu-\pr}m)=(f_\l g)_{-\mu-\pr}m$,
    \item [(2)] $f_{-\l-\pr}(g_\mu m)=(f_{-\mu-\pr}g)_{-\l-\pr+\mu} m $,
    \item [(3)] $f_{-\l-\pr}(g_{\mu-\pr}  m)=(f_{-\l-\pr+\mu}g)_{-\mu-\pr} m $.
\end{itemize}
\end{lem}

The proof follows from a straightforward computation.

\begin{defi}
Let $A$ be a conformal algebra.  A conformal linear map $d \in Cend(A)$ is a conformal anti-derivation on $A$ if for all $a,b \in A$, 
\begin{align}\label{anti-der}
    d_\l(a_\mu b)=-(d_\l a)_{\l+\mu} b-a_\mu(d_\l b).
\end{align}
\end{defi}

Now, we introduce the notion of Jacobi-Jordan conformal algebra. It generalizes Jacobi-Jordan algebras, which correspond to   $\l=0$. On the other hand, this variety can be viewed as   as a variety of commutative  Leibniz conformal algebras. 

\begin{defi}\label{JJ conformal alg}
A  Jacobi-Jordan conformal algebra is a conformal algebra $(A,\cdot_\lambda)$ such that
the following conditions hold:
\begin{align}
& a_\lambda b= b_{-\l-\pr}a, \label{commutativity} \\
&a_\l(b_\mu c)+(a_\l b)_{\l+\mu}c+b_\mu(a_\l c)=0 , \label{conf Jacobi identity}
\end{align}
for any $a,b,c \in A$. 
\begin{rmk}
If we take in the above definition, $\l=\mu=0$ and $\pr=0$, then we recover  the definition of a Jacobi-Jordan algebra. 
\end{rmk}

\end{defi}
We  consider a  conformal linear map $d_\l : A \to  Cend(A)$ defined by $d_\l(a)(b)=a_\l b$. Then, Eq. \eqref{conf Jacobi identity} is equivalent to say that $d_\l (a)$ is a conformal anti-derivation, for any $a \in A$.

\begin{ex}
Consider one of the simplest  (through important) examples of Jacobi-Jordan  conformal algebras.
Let $(A,\c)$ be a Jacobi-Jordan  algebra.  Then we  naturally define a  Jacobi-Jordan  conformal algebra $\textrm{Cur} A=\mathbb{F}[\pr]\otimes A$ with the $\l$-product
$$a_\l b=a\c b,~~~\forall a,b \in A.$$
$\textrm{Cur} A$ is called the current  Jacobi-Jordan  conformal algebra associated with $A$.
\end{ex}

\begin{ex}
Let $A$ be an associative commutative algebra and $(B,\c_\l)$ be a Jacobi-Jordan  conformal algebra.  Then the tensor product $A \otimes B$ becomes a Jacobi-Jordan  conformal algebra via
\begin{align*}
\pr(a \otimes x)=a \otimes \pr x, ~~~ (a\otimes x)_\l (b \otimes y)=(ab) \otimes (x_\l y).
\end{align*}

\end{ex}

\begin{defi}
An anti-associative conformal algebra is a conformal algebra $(A,\cdot_\lambda)$ satisfying
\begin{align}
(a_\l b)_{\l+\mu}c+ a_\l(b_\mu c)=0,
\end{align}
for all $a,b,c \in A$.
\end{defi}

Similarly as left symmetric conformal algebras are closely related to Lie conformal algebras, we introduce the notion of left anti-symmetric conformal algebras whose commutator gives rise  to a Jacobi-Jordan conformal algebra.

\begin{defi}
A conformal algebra $(A,\c_\l)$ is called  left anti-symmetric conformal algebra if for all $a,b ,c \in A$, one has
\begin{align}\label{left anti symm}
    (a_\l b)_{\l+\mu} c+ a_\l (b_\mu c) +     (b_\mu a)_{\l+\mu} c+ b_\mu (a_\l c)=0.
\end{align}
\end{defi}
It is obvious that every anti-associative conformal algebra is  left anti-symmetric.
A conformal algebra $(A, \cdot_\l)$ is said to be Jacobi-Jordan admissible if $(A,\circ_\l)$ is a Jacobi-Jordan conformal algebra, where $a\circ_\l b= a_\l b +b_{-\l-\pr} a$, for all $a,b \in A$. 

\begin{pro}
Let $(A,\cdot_\l)$ be a left anti-symmetric  conformal algebra. 
Then $A$ is a Jacobi-Jordan admissible conformal algebra. 

\end{pro}

\begin{proof}
Let $a,b,c \in A$. Then  $b\ci_{-\l-\pr} a= b_{-\l-\pr}a+a_\l b=a\ci_\l b$. On the other hand,
\begin{align*}
&a \ci_\l(b \ci _\mu c)+(a \ci_\l b)\ci_{\l+\mu} c+ b \ci_\mu(a \ci_\l c) \\
=&
a\ci_\l(b_\mu c+c_{-\mu-\pr} b)+(a_\l b+b_{-\l-\pr} a)\ci_{\l+\mu}c +b\ci_\mu(a_\l c+c_{-\l-\pr} a)  \\
=& a_\l(b_\mu c+c_{-\mu-\pr} b)+(b_\mu c+c_{-\mu-\pr} b)_{-\l-\pr }a + (a_\l b+b_{-\l-\pr} a)_{\l+\mu}c \\
+& c_{-\l-\mu-\pr}(a_\l b+b_{-\l-\pr} a)+b_\mu (a_\l c+c_{-\l-\pr} a)+(a_\l c+c_{-\l-\pr} a)_{-\mu-\pr}b \\
=& (a_\l b)_{\l+\mu}c+ a_\l(b_\mu c)+ (b_{-\l-\pr}a)_{\l+\mu} c+ b_\mu(a_\l c) \\
+& (a_\l c)_{-\mu-\pr} b+a_\l (c_{-\mu-\pr} b)+ (c_{-\l-\pr} a)_{-\mu-\pr} b+
c_{-\l-\mu-\pr} (a_\l b) \\
=&0. 
\end{align*}
Then $(A,\ci_\l)$ is a Jacobi-Jordan conformal algebra. 
\end{proof}

\begin{cor}
Any anti-associative conformal algebra is a Jacobi-Jordan admissible  conformal algebra.
\end{cor}

Similarly, we  have the following proposition. 
\begin{pro}
Any Jacobi-Jordan conformal algebra is Jacobi-Jordan admissible. 
\end{pro}

Now, we introduce a notion of a representation of a Jacobi-Jordan conformal algebra and give some examples. 
\begin{defi}
An $\mathbb{F}[\pr]$-module $M$ is a conformal module of
a Jacobi-Jordan conformal  algebra $A$  if there is an $\mathbb{F}$-linear map  $\pi: A \to Cend(M)$  satisfying
the following condition: For all $a, b  \in A$,
\begin{align}
& \pi(\pr a)_\l= -\l \pi(a)_\l ,\label{rep1}\\
&   \pi(a_\l  b)_{\l+\mu}=-\pi(a)_\l \pi(b)_\mu -\pi(b)_\mu \pi(a)_\l . \label{rep2}
\end{align}

\end{defi}

\begin{pro}
Given a finite dimensional complex Jacobi-Jordan algebra $A$.  Let $\rho : A  \to End(M)$  be a finite dimensional representation of $A$.  Then the free $\mathbb{C}[\pr]$-module $\mathbb{C}[\pr] \otimes M$   is a conformal module  of $Cur\ A$, where the module structure $\pi:  Cur\ A \to Cend(\mathbb{C}[\pr] \otimes M)$ is given by
\begin{align*}
\pi(f(\pr) \otimes a)_\l (g(\pr) \otimes m) =f(-\l)g(\l+\pr)\otimes  (\rho(a)m).
\end{align*}
\end{pro}

\begin{proof}
For any $a,b \in A, m \in M$ and $f(\pr), g(\pr), h(\pr) \in \mathbb{F}[\pr]$, we have 
\begin{align*}
\pi(f(\pr)a)_\l \pi(g(\pr)b)_\mu(h(\pr)m)=& \pi(f(\pr)a)_\l\Big(g(-\mu)h(\pr+\mu)\rho(b)m \Big) \\
=& f(-\l)g(-\mu)h(\pr+\l+\mu) \rho(a)\rho(b) m,
\end{align*}
and
\begin{align*}
\pi(g(\pr)b)_\mu \pi(f(\pr)a)_\l(h(\pr)m)=& \pi(g(\pr)b)_\mu\Big(f(-\l)h(\pr+\l)\rho(a)m \Big) \\
=& g(-\mu)f(-\l)h(\pr+\l+\mu) \rho(b)\rho(a) m.
\end{align*}
Then 
\begin{align*}
\pi(f(\pr)a_\l g(\pr)b)_{\l+\mu} (h(\pr)m) =&
\pi(f(-\l)g(\pr+\l)ab)_{\l+\mu}(h(\pr)m) \\
=& f(-\l)g(-\mu)h(\l+\mu+\pr) \rho(ab)m \\
=& -f(-\l)g(-\mu)h(\l+\mu+\pr) \big(\rho(a)\rho(b)m+\rho(b)\rho(a) m \big) \\
=&-\pi(f(\pr)a)_\l \pi(g(\pr)b)_\mu(h(\pr)m)- \pi(g(\pr)b)_\mu \pi(f(\pr)a)_\l(h(\pr)m).
\end{align*}
Hence the proof.

\end{proof}

\begin{pro}
Let $(A,\c_\l)$ be a Jacobi-Jordan conformal algebra, $M$ an $A$-conformal module and $\pi: A \to \mathbb{F}[\pr] \otimes Cend(M): \ r \mapsto \pi(r)_\l $ the corresponding representation.  Define the following $\l$-product  on $A \oplus M$ :
\begin{align}\label{semi direct product}
     & (a+ m)_\l (b+n)=
       a_\l b+
      \pi(a)_\l n + \pi(b)_{-\l-\pr} m  .
\end{align}
Then $A \oplus M$  becomes a Jacobi-Jordan  conformal algebra with respect to the above $\l$-product,
called the semidirect product of $A$ and $M$, and denoted as $A\ltimes M$. Moreover, $M$ is an abelian ideal of  $A\ltimes M$.
\end{pro}

\begin{proof}
For any $a,b,c  \in A$ and $m,n,p  \in M$, we have
\begin{align*}
    \pr (a+ m )  _\l  (b+ n )=&
 \pr a_\l  b + \pi(\pr a)_\l  n + \pi(b)_{-\l-\pr} \pr m \\
=&-\l (  a_\l b  +   \pi(a)_\l n +  \pi(b)_{-\l-\pr} m )
=-\l   (  a +  m )  _\l  (b + n ).
\end{align*}
Using a similar computation, one can easily check that
\begin{align*}
&   (a+ m  ) _\l  \pr (b+ n)  =(\l+\pr)     (a +  m ) _\l (b +n).
\end{align*}
Moreover,
\begin{align*}
&(  b +  n   )_{-\l-\pr}  ( a + m )=
(   b_{-\l-\pr} a +  \pi(b)_{-\l-\pr}m+\pi(a)_\l m   )
= (   a +  m   )_\l  ( b + n ).
\end{align*}
In addition,
\begin{align*}
   ( a +m )_\l  \Big( (b + n) _\mu (  c + p ) \Big) 
=& ( a +m) _\l  ( b_\mu c)  +  \pi(b)_\mu p + \pi(c)_{-\mu-\pr} n  \\
=& (    a_\l(b_\mu c) +   \pi(a)_\l (\pi(b)_\mu p)+ \pi(a)_\l ( \pi(c)_{-\mu-\pr} n) +\pi(b_\mu c)_{-\l-\pr}m   ),
\end{align*}

\begin{align*}
   &\Big( (a +m) _\l  ( b + n )\Big)_{\lambda+\mu}  (  c + p ) \\
=&   ( a_\l b  +  \pi(a)_\l n + \pi(b)_{-\l-\pr} m ) _{\lambda+\mu}( c +p )\\
=& \big(    (a_\l b)_{\l+\mu} c +   \pi(a_\l b)_{\l+\mu} p+ \pi(c)_{-\l-\mu-\pr} (\pi(a)_\l n) + \pi(c)_{-\l-\mu-\pr}(\pi(b)_{-\l-\pr} m)   \big)
\end{align*}
and 

\begin{align*}
   ( b +n )_\mu  \Big(( a + m )_\l (  c + p ) \Big) 
=& (    b_\mu(a_\l c) +   \pi(b)_\mu (\pi(a)_\l p)+ \pi(b)_\mu  (\pi(c)_{-\l-\pr} m) +\pi(a_\l c)_{-\mu-\pr}n  ). 
\end{align*}
Then 
\begin{align*}
&  ( a +m )_\l  \Big(( b + n )_\mu (  c + p ) \Big)
+    \Big(( a +m )_\l  ( b + n )\Big)_{\lambda+\mu}  (  c + p ) 
+ ( b +n )_\mu  \Big(( a + m )_\l (  c + p ) \Big)  \\
=& ( a_\l(b_\mu c)+(a_\l b)_{\l+\mu}c+b_\mu(a_\l c) + X+ Y+Z ) = (0+ X+Y+Z ),
\end{align*}
where
\begin{align*}
&X=\pi(a)_\l(\pi(b)_\mu p)+\pi(a_\l b)_{\l+\mu} p+ \pi(b)_\mu(\pi(a)_\l p ),\\
&Y=\pi(a)_\l ( \pi(c)_{-\mu-\pr} n)+\pi(c)_{-\l-\mu-\pr} (\pi(a)_\l n)+\pi(a_\l c)_{-\mu-\pr}n,\\
&Z=\pi(b_\mu c)_{-\l-\pr}m + \pi(c)_{-\l-\mu-\pr}(\pi(b)_{-\l-\pr} m)+\pi(b)_\mu  (\pi(c)_{-\l-\pr} m).
\end{align*}
According to Lemma \ref{Lemme one} and identity \eqref{rep2}, one can check that $X=Y=Z=0$.
\end{proof}

\begin{pro}
Let $(M,\pi)$ be a finite dimensional representation of a Jacobi-Jordan conformal algebra $(A,\cdot_\l)$. Let $\pi^*$ be an $\mathbb{F}[\pr]$-module homomorphism from $A$ to $Cend(M^{*c})$ defined by
\begin{align}
 ( \pi^*(a)_\l f)_\mu m= f_{\mu-\l}(\pi(a)_\l m), \ \forall a \in A, f \in M^{*c}, m \in M.
\end{align}
Then $(M^{*c},\pi^*)$  is a representation of $A$. 
\end{pro}

\begin{proof}
Let $a,b \in A$, $m \in M$ and $f \in M^{*c}$. Then we have
\begin{align*}
 \big(\pi^*(\pr a)_\l f\big)_\mu m= f_{\mu-\l}\big(\pi(\pr a)_\l m\big)= -\l f_{\mu-\l} \big(\pi(a)_\l m\big)=-\l  \big(\pi^*( a)_\l f\big)_\mu m.    
\end{align*}
Then \eqref{rep1} holds. To prove \eqref{rep2}, we compute as follows
\begin{align*}
\big(\pi^*(a_\l b)_{\l+\mu} f \big)_\nu m=&f_{\nu-\l-\mu}\big(\pi(a_\l b)_{\l+\mu}m \big) \\
=& -f_{\nu-\l-\mu}\big(\pi(a)_\l \pi(b)_\mu m \big)  -f_{\nu-\l-\mu}\big(\pi(b)_\mu \pi(a)_\l m \big).
\end{align*}
On the other hand, we have
\begin{align*}
  &  \pi^*(a)_\l\big(\pi^*(b)_\mu f \big)_\nu m= \big(\pi^*(b)_\mu f \big)_{\nu-\l} (\pi(a)_\l m) =f_{\nu-\l-\mu}\big( \pi(b)_\mu \pi(a)_\l m \big),\\
  &  \pi^*(b)_\mu\big(\pi^*(a)_\l f \big)_\nu m= \big(\pi^*(a)_\l f \big)_{\nu-\mu} (\pi(b)_\mu m) =f_{\nu-\l-\mu}\big( \pi(a)_\l \pi(b)_\mu m \big). 
\end{align*}
Therefore
\begin{align*}
\big(\pi^*(a_\l b)_{\l+\mu} f \big)_\nu m=- \pi^*(a)_\l\big(\pi^*(b)_\mu f \big)_\nu m- \pi^*(b)_\mu\big(\pi^*(a)_\l f \big)_\nu m,   
\end{align*}
which implies that  $(M^{*c},\pi^*)$  is a representation of $A$.
\end{proof}

\subsection{$\mathcal{O}$-operators on Jacobi-Jordan conformal algebras}

\begin{defi}
Let $(A,\c_\l)$ be a Jacobi-Jordan conformal algebra and $\pi: A \to Cend(M)$ be a representation. An $\mathbb{F}[\pr]$-module homomorphism $T: M \to A$ 
satisfying
\begin{align}\label{O-Operator}
& Tu_\l Tv= T\big(\pi(Tu)_\l v+ \pi(Tv)_{-\l-\pr} u \big ),\ \ \forall u,v\in M,  
\end{align}
 is called an   $\mathcal{O}$-operator.
\end{defi}

If $M=A$ and $\pi$ is the adjoint representation, then $T$ is  called a Rota-Baxter operator  of weight $0$ on $A$ denoted by $R$ and the condition \eqref{O-Operator} can be rewritten as
\begin{align}\label{RotaBaxter}
& R(a)_\l R(b)= R\big(R(a)_\l b+ a_\l R(b) \big ),\ \ \forall a,b\in A.   
\end{align}
More generally, we can extend the notion of weighted Rota-Baxter operators of Jacobi-Jordan algebras to  conformal case by considering the following weighted Rota-Baxter identity:
\begin{align}\label{WeightedRotaBaxter}
& R(a)_\l R(b)= R\big(R(a)_\l b+ a_\l R(b)\big)  +\alpha R(a_\l b),\ \ \forall a,b\in A.   
\end{align}

\begin{thm}\label{ALS conf on module}
Let $A$ be a Jacobi-Jordan conformal  algebra and $\pi : A \to  Cend(M)$  be a representation of $A$. 
Suppose $T : M \to A$ is an  $\mathcal{O}$-operator associated with $\pi$. Then, the
following $\l$-product:
\begin{align}
u \ast_\l  v = \pi(T u)_\l v,\ \ \forall u, v \in  M, 
\end{align}
endows $M$ with a  left anti-symmetric conformal algebra structure. 
Therefore, $M$ is a Jacobi-Jordan  conformal algebra which is the sub-adjacent Jacobi-Jordan  conformal algebra of this left anti-symmetric
conformal algebra, and $T : M  \to A$ is a homomorphism of Jacobi-Jordan  conformal algebras.
Moreover, $\{T (M) \subset  A \}$ is a Jacobi-Jordan  conformal subalgebra of $A$ and there is also a natural
left anti-symmetric conformal algebra structure on $T (M)$ defined as follows:
$$T (u)_\l T (v) = T (u \ast_\l  v) = T (\pi(T u)_\l v),\ \ \forall u, v \in  M.$$
In addition, the sub-adjacent Jacobi-Jordan  conformal algebra of this left anti-symmetric conformal
algebra is a subalgebra of $A$ and $T : M \to  A$ is a homomorphism of left anti-symmetric
conformal algebras.

\end{thm}

\begin{proof}
For any $u,v,w \in M$, we have
\begin{align*}
 &\pr u \ast_\l v=\pi(T(\pr u))_\l v=\pi(\pr Tu)_\l v=-\l \pi(Tu)_\l v=-\l u \ast_\l v,\\ 
 & u \ast_\l \pr v=\pi(Tu)_\l (\pr v)=(\pr+\l )\pi(Tu)_\l v=(\pr+\l )u \ast_\l v. 
\end{align*}
Furthermore we have 
\begin{align*}
&(u \ast_\l v)\ast_{\l+\mu}w+ u\ast _\l (v\ast_\mu w) \\ 
=&(\pi(Tu)_\l v)\ast_{\l+\mu}w +u\ast_\l(\pi(Tv)_\mu w) \\
=& \pi(T(\pi(Tu)_\l v))_{\l+\mu}w+ \pi(Tu)_\l \pi(Tv)_\mu w\\
=& \pi(T(u)_\l T(v))_{\l+\mu}w-\pi(T(\pi(Tv)_{-\l-\pr}u)_{\l+\mu} w -\pi(T(u)_\l bT(v))_{\l+\mu}w-\pi(Tv)_\mu \pi(Tu)_\l w \\
=&-\pi(T(\pi(Tv)_{-\l-\pr}u)_{\l+\mu} w -\pi(Tv)_\mu \pi(Tu)_\l w \\
=&-\pi(T(\pi(Tv)_\mu u)_{\l+\mu} w -\pi(Tv)_\mu \pi(Tu)_\l w \\
=& -(v \ast_\mu u)\ast_{\l+\mu}-v \ast_\mu (u \ast_\l w).
\end{align*}
Hence $(M, \ast_\l)$ is a left ant-symmetric conformal algebra.
\end{proof}

\begin{cor}
Let $A$ be a Jacobi-Jordan  conformal algebra and $T : A \to A$  be a Rota–Baxter
operator of weight zero. Then, there is a left anti-symmetric conformal algebra structure
on $A$  with the following $\l$-product:
$$ a \circ_\l b= T(a)_\l b,\quad a,b \in A.$$
\end{cor}

\begin{cor}
Let $(A,\cdot_\l)$ be a Jacobi-Jordan  conformal algebra. There is a compatible left anti-symmetric  conformal algebra structure on $A$ if and only if there exists a bijective $\mathcal{O}$-operator
$T : M \to A$ associated with some representation  $(M,\pi)$ of $A$. 
\end{cor}

\begin{proof}
Suppose that there exists a bijective $\mathcal{O}$-operator $T: M \to A$  of $(A,\cdot_\l)$ associated
with a representation $(M,\pi)$. 
Then by Theorem \ref{ALS conf on module},  with a straightforward checking, we have that
$$a \ast_\l b=T(\pi(a)_\l T^{-1}(b)),\ \ \forall a,b\in A,$$ defines a compatible left anti-symmetric conformal algebra structure on $A$. 

Conversely, suppose that there is a compatible left anti-symmetric  conformal algebra structure 
on $A$. Then the identity map $Id: A \to A$ is  a bijective $\mathcal{O}$-operator of $A$
associated with  the adjoint representation. 

\end{proof}

\subsection{Symplectic Jacobi-Jordan  conformal algebras}
Next, we introduce a definition of symplectic Jacobi-Jordan  conformal algebras and show a relationship with left anti-symmetric conformal algebras. 
First, let us recall  some notations about conformal bilinear forms. 

Let $V$ be an   $ \mathbb{F}[\l]$-module.   A conformal bilinear form on $V$ is an
 $\mathbb{F}$-linear map $\phi_\l: V \otimes V \to \mathbb{F}[\l]$    satisfying the following conditions
\begin{align}
    & \phi_\l (\pr u,v)=-\l \phi_\l (u,v),\quad 
         \phi_\l ( u,\pr v)=\l \phi_\l (u,v).
\end{align} 
 A conformal bilinear form is called  skew-symmetric if
 $\phi_\l(u,v)=-\phi_{-\l}(v,u)$. 
 Suppose that there is a conformal bilinear form on a $\mathbb{F}[\pr]$-module $V$. Then, we
have an $\mathbb{F}[\pr]$-module homomorphism  $T : V \to V^{*c},\  u \mapsto  T_u$  given by
$$ (T_u)_\l  v= \phi_\l (u,v),\ u,v \in V.$$
A  conformal bilinear form is called non-degenerate if $T$ is an isomorphism of $\mathbb{F}[\pr]$-modules between $V$ and $V^{*c}$.

\begin{defi}
A Jacobi-Jordan  conformal algebra $A$ is called a symplectic Jacobi-Jordan conformal
algebra if there is a non-degenerate skew-symmetric conformal bilinear form $\phi_\l$ on $A$
satisfying
\begin{align} 
& \phi_\l (a_,b_\mu c)  +\phi_\mu(b,a_\l c)+
\phi_{\l+\mu} (a_\l b,c)=0  .     
\end{align}
\end{defi}

\begin{ex}
Let $(A,\phi)$ be  a symplectic Jacobi-Jordan algebra. Then, $(Cur A, \phi_\l)$ is a
symplectic Jacobi-Jordan conformal algebra where $\phi_\l$
is defined by
$$ \phi_\l(f(\pr) a,g(\pr) b)=f(-\l)g(\l)\phi(a,b),\ \forall f(\pr), g(\pr) \in \mathbb{F}[\pr], a,b \in A.$$
\end{ex}

The following result relates left anti-symmetric conformal algebras with symplectic Jacobi-Jordan conformal algebras.

\begin{pro}
Let $(A,\cdot_\l, \phi_\l)$ be a symplectic Jacobi-Jordan conformal algebra.  Then there exists a compatible left anti-symmetric conformal  algebra structure $\circ_\l$ on $A$ given by
\begin{align}
    & \phi_\mu (a \circ_\l b,c)=\phi_{\mu-\l}(b,a_\l c),\ \forall a,b,c \in A.
\end{align}
\end{pro}

\begin{proof}
Straightforward.
\end{proof}

\section{Quadratic Jacobi-Jordan conformal algebras and mock-Gel'fand Dorfman bialgebras}\label{Characterization}
\subsection{Mock-Gel'fand Dorfman bialgebras}
In this section, we introduce a new class of algebras which will be useful to characterize quadratic Jacobi-Jordan conformal algebras. This class will be called mock-Gel'fand Dorfman bialgebras. We first provide some relevant definitions and constructions.  
\begin{defi}
An anti-Novikov algebra is a pair $(A,\circ)$, where $A$ is a vector space and $\circ$ is an operation on $A$ satisfying  the following axioms:
\begin{align}
    & a \circ (b \circ c)=-b \circ (a \circ c),\\
    & (a,b,c)^-_\circ + (a,c,b)^-_\circ =0,
\end{align}
for any $a,b,c \in A$ and  where $(a,b,c)^-_\circ=(a\circ b)\circ c+ a \circ (b \circ c)$. 
\end{defi}

\begin{lem}\label{anti-Nov to JJ alg}
Let $(A,\circ)$ be an anti-Novikov algebra. Then $(A,\ast)$ is a Jacobi-Jordan  algebra, where, for any $a,b \in A$, $a \ast b=a \circ b + b \circ a$. 
\end{lem}

\begin{pro}

Let $(A,\cdot)$ be an anti-commutative anti-associative algebra and $d: A \to A$ be a derivation on $A$. Then $(A,\circ)$ is an anti-Novikov algebra, where
$$ a \circ b= d(a)\cdot b,\ \forall a,b \in A.$$

\end{pro}

\begin{proof}
Let $a,b, c \in A$. Then we have
\begin{align*}
    a \circ (b \circ c)=&d(a)\cdot (d(b)\cdot c) 
    = -d(a)\cdot (c \cdot d(b)) 
    = (d(a) \cdot c)\cdot d(b) \\
    =& -d(b) \cdot (d(a) \cdot c) 
    = -b \circ (a \circ c).
\end{align*}
On the other hand,
\begin{align*}
&(a \circ b) \circ c+a \circ (b \circ c)=
d(d(a) \cdot b)\cdot c+d(a)\cdot (d(b) \cdot c) \\
=& (d^2(a)\cdot b)\cdot c+(d(a)\cdot d(b))\cdot c+ d(a) \cdot (d(b) \cdot c) \\
=& -d^2(a)\cdot (b \cdot c) =d^2(a)\cdot (c \cdot b).
\end{align*}
Since $b$ and $c$ have symmetric role, then we have 
\begin{align*}
&(a \circ c) \circ b+a \circ (c \circ b)=
=-d^2(a)\cdot (c \cdot b). 
\end{align*}
Therefore, $(A,\circ)$ is an anti-Novikov algebra. 
\end{proof}

Now, we introduce the notion of mock-Gel'fand Dorfman bialgebra. 

\begin{defi}
A mock-Gel'fand Dorfman bialgebra is a vector space $A$ with two operations $\circ$ and $\ast$, such that  $(A,\circ)$ is an anti-Novikov algebra, $(A,\ast)$ is a Jacobi-Jordan algebra and the following compatibility condition holds (for any $a,b,c \in A$):
\begin{align}\label{compatibility cond}
  &  a \circ (b \ast c)+ a \ast (b \circ c)+ (a \ast b) \circ c+ b \circ (a \ast c)+ b \ast (a \circ c)=0.
\end{align}
Such an algebra will be denoted by $(A,\ast,\circ)$. 
\end{defi}
In the sequel, we shall
simply call a mock-Gel’fand-Dorfman bialgebra a mock-GD bialgebra.\\
The following construction of mock-GD bialgebras is from anti-Novikov algebras, analogous to the
fundamental construction of Gel’fand-Dorfman bialgebras from Novikov algebras via the commutator bracket. 

\begin{thm}
Let $(A,\circ)$ be an anti-Novikov algebra. Define a new product $\ast$ on $A$ by
\begin{equation}
    a \ast b= a \circ b + b \circ a,\ \text{for}\ a,b, \in A.
\end{equation}
Then $(A, \ast,\circ)$ is a mock-GD bialgebra. 
\end{thm}

\begin{proof}
By Lemma \ref{anti-Nov to JJ alg}, $(A,\ast)$ is a Jacobi-Jordan algebra. So we only, prove the compatibility condition \eqref{compatibility cond}.  For any $a,b,c \in A$, we have
\begin{align*}
    &  a \circ (b \ast c)+ a \ast (b \circ c)+ (a \ast b) \circ c+ b \circ (a \ast c)+ b \ast (a \circ c)  \\
    =& a \circ (b \circ c)+ a \circ (c \circ b)+  a \circ (b \circ c) + (b \circ c) \circ a +  (a \circ b) \circ c \\
    &+  (b \circ a) \circ c +  b \circ (a \circ c)+  b \circ (c \circ a)+  b \circ (a \circ c)+ (a \circ c) \circ b  \\
    =& (a \circ b) \circ c+  a \circ (b \circ c)+  (a \circ c )\circ b +  a \circ (c \circ b) \\
    &+  (b \circ a) \circ c+  b \circ (a \circ c)+ (b \circ c) \circ a +   b \circ (c \circ a)\\
    &+  a \circ (b \circ c) +  b \circ (a \circ c)\\=&0.
\end{align*}
\end{proof}

\begin{pro}
Let $(A,\cdot)$ be an anti-commutative anti-associative algebra and $d:A \to A$ be a derivation. Then $(A,\ast,\circ)$ is a mock-GD algebra, where $\ast$ and $\circ$ are defined by
\begin{align}
a \circ b=d(a) \cdot b,\quad a \ast b=d(a) \cdot b+  d(b) \cdot a,\ \forall a,b \in A.    
\end{align}
\end{pro}
\begin{ex}\label{example with derivation}
Let $A$ be a $3$-dimensional vector space with basis $\{e_1,e_2,e_3\}$. Define the following product on $A$:
$ e_1\cdot e_2=-e_2\cdot e_1=e_3.$  Then it is easy to check that $(A,\cdot)$ is an anti-commutative anti-associative algebra. Consider the linear map $d:A \to A$ defined by:
$$ d(e_1)=e_1,\quad d(e_2)=2e_2,\quad d(e_3)=3e_3.$$
Then $d$ is a derivation on $A$. Therefore, according to the above proposition, $(A,\ast,\ci)$ is a mock-GD bialgebra, where
$$e_i \ci e_i=e_i\ci e_i=0,\quad e_1\ci e_2=e_3,\quad e_2\ci e_1=-2e_3,\quad e_1\ast e_2=e_2\ast e_1=e_3.$$
\end{ex}

\subsection{Quadratic Jacobi-Jordan conformal algebras and their characterization}
\begin{defi}
Let $A$ be a Jacobi-Jordan conformal algebra. If there exists a vector space $V$ such that $A=\mathbb{F}[\pr]V$ is an $\mathbb{F}[\pr]$-module over $V$ and for all $a,b\in V$ the $\l$-product is of the following form:
$$ a_\l b=\pr u+ \l v+ w, $$
where $u,v, w \in V$, then $A$ is called a quadratic Jacobi-Jordan conformal algebra. 
\end{defi}
 We have the following characterization result.
\begin{thm}

Let $V$ be a vector space equipped with two operations $\ast$ and $\circ$. Let $A=\mathbb{F}[\pr]V$ be  the free  $\mathbb{F}[\pr]$-module over $V$. Define the $\l$-product $\cdot_\l: A \otimes A \to A[\l]$ on $A$  by
\begin{align}
 u_\l v=u \ast v+ \pr(u \circ v)+\l (u \circ v-v \circ u),\ \ u,v \in A.
\end{align}
Then $(A,\cdot_\l\cdot)$ is a quadratic Jacobi-Jordan conformal algebra if and only if $(V,\ast, \circ)$ is a mock-GD algebra. 

\end{thm}

\begin{proof}
Suppose that $A$ is a quadratic Jacobi-Jordan conformal algebra. By its definition, we set  
\begin{align}
 u_\l v=u \ast v+ \pr(u \circ v)+\l (u \circ v-v \circ u),\ \forall u,v \in V,
\end{align}
where $\ast$ and $\circ$ are two $\mathbb{F}$-bilinear maps from $V \times V \to V$. 
Next, we consider axioms \eqref{commutativity} and \eqref{conf Jacobi identity}. For any $a,b, c \in A$, we have
\begin{align*}
b_{-\l-\pr}a=&b \ast a +\pr(b \circ a)-(\l+\pr)(b \circ a-a \circ b) \\
=& b \ast a+ \pr(a \circ b)+ \l (a \circ b-b \circ a).
\end{align*}
In addition, we have 
$a_\l b=a \ast b + \pr(a \circ b)+ \l (a \circ b-b \circ a)$.    
Since \eqref{commutativity} holds, then 
\begin{align}
& a \ast b=b \ast a. \label{equiv0}
\end{align}
On the other hand,
\begin{align*}
 & a_\l(b_\mu c)=a_\l(b \ast c+ \pr(b \circ c)+\mu(b \circ c-c \circ b))  \\
 =&a_\l(b \ast c)+ (\l+\pr)a_\l (b \circ c)+\mu a_\l(b \circ c-c \circ b) \\
 =&a \ast (b \ast c)+ \pr(a \circ (b \ast c))+\l(a \circ (b \ast c)-(b \ast c)\circ a) \\
 +&(\l+\pr)a \ast (b \circ c)+(\l+\pr)\pr (a \circ (b \circ c))+\l(\l+\pr)(a \circ (b \circ c)-(b \circ c)\circ a) \\
 +&\mu a \ast(b \circ c-c \circ b)+\mu \pr a \circ (b \circ c-c \circ b)+\mu \l(a \ci (b \ci c-c \ci b)-  (b \ci c-c \ci b)\ci a) \\
 =& a \ast (b \ast c)+ \pr \big(a \ci(b \ast c)+ a \ast(b \ci c)\big)+\l\big(a\ci (b \ast c)-(b \ast c)\ci a+a \ast (b \ci c) \big) \\
 +& \mu a \ast(b \ci c-c \ci b)+ \l \pr\big(2a \ci (b \ci c)-(b \ci c)\ci a \big) +\l \mu \big(a \ci (b \ci c-c \ci b)-(b \ci c-c \ci b)\ci a \big)\\
 +&\mu \pr a \ci (b \ci c-c \ci b)+\pr^2 a \ci (b \ci c)+ \l^2 \big(a \ci (b \ci c)-(b \ci c) \ci a\big).
\end{align*}
Similarly, we get
\begin{align*}
&b_\mu(a_\l c)= b \ast (a \ast c)+ \pr \big(b \ci(a \ast c)+ b \ast(a \ci c)\big)+\mu\big(b\ci (a \ast c)-(a \ast c)\ci b+b \ast (a \ci c) \big) \\
 +& \l b \ast(a \ci c-c \ci a)+ \mu \pr\big(2b \ci (a \ci c)-(a \ci c)\ci b \big) +\l \pr b \ci (a \ci c-c \ci a)\\
 +& \l \mu \big(b \ci (a \ci c-c \ci a)-(a \ci c-c \ci a)\ci b \big)+\pr^2 b \ci (a \ci c)+ \mu^2 \big(b \ci (a \ci c)-(a \ci c) \ci b\big).   
\end{align*}
Furthermore, we can obtain
\begin{align*}
&(a_\l b)_{\l+\mu}c=(a\ast b)\ast c+ \pr (a \ast b)\ci c+ \l \big((a\ast b)\ci c-c \ci (a \ast b) -(b \ci a)\ast c \big)  \\
+&\mu \big((a \ast b)\ci c-c \ci (a \ast b)-(a \ci b)\ast c \big)-\l \pr (b \ci a) \ci c-\mu \pr (a \ci b)\ci c \\
+& \l \mu \big((b \ci a-a\ci b)\ci c+c \ci (a \ci b+b \ci a) \big)-\mu^2\big((a\ci b)\ci c-c\ci(a\ci b) \big) \\
+& \l^2 \big( c \ci (b\ci a)-(b\ci a)\ci c\big). 
\end{align*}
According to the identity \eqref{conf Jacobi identity} and comparing the coefficients of $\l^2,\mu^2,\pr^2,\l\mu, \l \pr, \mu \pr,\l,\mu,\pr$ and $\l^0\mu^0\pr^0$, we obtain
\begin{align}
&0=a \ast (b \ast c)+(a \ast b)\ast c+b \ast (a \ast c),\label{equiv1}\\
&0=a \ci(b \ast c)+a\ast (b\ci c)+b \ci(a\ast c)+b\ast(a\ci c)+(a\ast b)\ci c,\label{equiv2}\\
&0=a\ci(b\ast c)-(b\ast c)\ci a+a\ast (b\ci a) +b \ast(a\ci c-c\ci a) \nonumber\\
& \qquad + (a\ast b)\ci c-c \ci (a\ast b)-(b\ci a)\ast c,\label{equiv3}\\
&0=a\ast (b \ci c-c\ci b)+b\ci (a\ast c)-(a\ast c)\ci b+b\ast(a\ci c)\nonumber \\
&\qquad  +(a\ast b)\ci c-c \ci (a\ast b)-(a\ci b)\ast c,\label{equiv4}\\
&0=2a \ci (b \ci c)-(b\ci c)\ci a+b\ci (a\ci c-c \ci a)-(b\ci a)\ci c,\label{equiv5}\\
&0=a \ci(b\ci c-c\ci b)+2b\ci(a\ci c)-(a\ci c)\ci b-(a\ci b)\ci c,\label{equiv6}\\
&0=a \ci (b \ci c-c \ci b)-(b \ci c-c \ci b)\ci a+b \ci (a \ci c-c \ci a)\nonumber\\
&\quad -(a \ci c-c \ci a)\ci b
+(b \ci a-a\ci b)\ci c+c \ci (a \ci b+b \ci a),\label{equiv7}\\
&0=a\ci(b\ci c)-(b\ci c)\ci a+c\ci(b\ci a)-(b\ci a)\ci c,\label{equiv8}\\
&0=b\ci (a\ci c)-(a\ci c)\ci b+c\ci(a\ci b)-(a\ci b)\ci c,\label{equiv9}\\
&0=a \ci (b\ci c)+b \ci (a \ci c).\label{equiv10}
\end{align}
By \eqref{equiv0} and \eqref{equiv1}, we deduce that $(V,\ast)$ is a Jacobi-Jordan algebra. Note that  \eqref{equiv8} and \eqref{equiv9} are the same. So if we take, for example \eqref{equiv9} with \eqref{equiv10}, we get that $(A,\circ)$ is an anti-Novikov algebra. On the other hand, identities \eqref{equiv5}, \eqref{equiv6} and \eqref{equiv7} are equivalent to \eqref{equiv9}.
Furthermore, equalities \eqref{equiv3} and \eqref{equiv4} are equivalent to \eqref{equiv2} which is just the compatibility condition \eqref{compatibility cond}. Hence, $(A,\ast,\circ)$ is a mock-GD bialgebra. 

Conversely, if $(A,\ast,\circ)$ is a mock-GD bialgebra, then by  a  similar computation, we can prove according to the above discussions that $(A,\cdot_\l)$ is a Jacobi-Jordan conformal algebra. 
The proof is finished. 

\end{proof}

\begin{ex}
Consider the mock-GD bialgebra constructed in Example \ref{example with derivation}. Then there is an associated Jacobi-Jordan conformal algebra whose $\l$-product is given by
$$ e_{1\ \l} e_2=(\pr+3\l-1)e_3,\quad e_{2\ \l} e_1=-(2\pr+3\l+1)e_3,$$
where the non given products are zeros.   
\end{ex}

\section{The $\mathbb{C}[\pr]$-split extending structures problem}\label{ExtendingStructures}
In this section, the  $\mathbb{C}[\pr]$-split extending structures problem for Jacobi-Jordan conformal algebras is studied. We introduce a definition of unified product of a given Jacobi-Jordan conformal
algebra $J$ and a given  $\mathbb{C}[\pr]$-module $K$. This product includes some other interesting products of Jacobi-Jordan conformal algebras such as twisted product.
Using this product, a cohomological type object is constructed to provide a theoretical answer to the  $\mathbb{C}[\pr]$-split extending structures problem.

We introduce the following definition which is relevant  for studying
the $\mathbb{C}[\partial]$-split extending structures problem.

\begin{defi}
Let $(J, \cdot_\lambda)$ be a Jacobi-Jordan conformal algebra, $K$ a $\mathbb{C}[\partial]$-module and $E=J\oplus K$, where the direct sum is the sum of $\mathbb{C}[\partial]$-modules. Let $\varphi: E\rightarrow E$ be a $\mathbb{C}[\partial]$-module homomorphism. We consider  the following
diagram:
$$\xymatrix{ {J}\ar[d]^{Id}\ar[r]^{i}& {E}\ar[d]^{\varphi}\ar[r]^{\pi} & {K}\ar[d]^{Id} \\
{J}\ar[r]^{i}&{E}\ar[r]^{\pi} &{K} }$$
where $\pi: E\rightarrow K$ is the natural projection of $E=J\oplus K$ onto $K$ and
$i: J\rightarrow E$ is the inclusion map. If the left square
(resp. the right square) of the  above diagram are commutative,
we say that
$\varphi: E\rightarrow E$ \emph{stabilizes} $J$ (resp. \emph{co-stabilizes} $K$).

Let $\ast_\lambda$ and $\square_\lambda$ be two Jacobi-Jordan conformal algebra structures on $E$ both containing $J$ as a Jacobi-Jordan conformal subalgebra.
If there exists a Jacobi-Jordan conformal algebra isomorphism $\varphi: (E, \ast_\lambda )\rightarrow (E,\square_\lambda)$ which stabilizes
$J$, then $\ast_\lambda $ and $\square_\lambda$ are called \emph{equivalent}. In this case, we denote it by
$(E,\ast_\lambda)\equiv (E,\square_\lambda)$.
If there exists a Jacobi-Jordan conformal algebra isomorphism $\varphi: (E, \ast_\lambda )\rightarrow (E,\square_\lambda)$ which stabilizes
$J$ and co-stabilizes $K$, then $\ast_\lambda$ and $\square_\lambda$ are called \emph{cohomologous}. In this case, we denote it by
$(E,\ast_\lambda)\approx (E,\square_\lambda)$.
\end{defi}

Obviously, $\equiv $ and $\approx $ are equivalence relations on the set of all Jacobi-Jordan conformal algebra structures on
$E$ containing $J$ as a Jacobi-Jordan conformal subalgebra. We denote  the set
of all equivalence classes via $\equiv $ (resp. $\approx $) by $\text{CExtd}(E,J)$ (resp. $\text{CExtd}^{'}(E,J)$). It is easy to see that $\text{CExtd}(E,J)$ is the classifying object of the
$\mathbb{C}[\partial]$-split extending structures problem, and there exists a canonical projection $\text{CExtd}^{'}(E,J)\twoheadrightarrow \text{CExtd}(E,J) $.

\subsection{Unified products of Jacobi-Jordan conformal algebras}
In this section, a unified product of Jacobi-Jordan conformal algebras is defined and used to provide a theoretical answer to the $\mathbb{C}[\partial]$-split extending structures problem for Jacobi-Jordan conformal algebras.

\begin{defi}
Let $(J,\cdot_\lambda)$ be a Jacobi-Jordan conformal algebra  and $K$ a $\mathbb{C}[\partial]$-module. An \emph{extending datum} of $J$ by $K$ is
a system $\mathcal U(J, K)=\{ \lhd_\lambda, \rhd_\lambda,  \omega_\lambda,  \circ_\lambda \}$ consisting of four conformal bilinear maps 
\begin{eqnarray*}
&\lhd_\lambda: K\times J\rightarrow K[\lambda],&~~~\rhd_\lambda: K\times J\rightarrow J[\lambda],\\
&\omega_\lambda: K\times K \rightarrow J[\lambda],&~~~\circ_\lambda: K\times K \rightarrow K[\lambda].
\end{eqnarray*}
Let $\mathcal U(J, K)=\{\lhd_\lambda, \rhd_\lambda,  \omega_\lambda, \circ_\lambda  \}$ be an extending datum of $J$ by $K$. Denote by  $J\natural K$ the $\mathbb{C}[\partial]$-module
$J\oplus K$ with the natural $\mathbb{C}[\partial]$-module action: $\partial (a, x)=(\partial a, \partial x)$ and the bilinear map $
\ast_\lambda : (J\oplus K)\times (J\oplus K)\rightarrow (J\oplus K)[\lambda]$ defined by
{\small
\begin{align}
(a,x)\ast_\lambda (b,y)
=(a_\lambda b+x\rhd_\lambda b+y \rhd_{-\l-\pr}a +\omega_\lambda(x,y),x\lhd_\lambda b+y \lhd_{-\l-\pr}a +x\circ_\lambda y), \label{ast product}
\end{align} }
for all $a$, $b\in J$, $x$, $y\in K$. Since $\lhd_\lambda$, $\rhd_\lambda$,  $\omega_\lambda$ and $\circ_\lambda$ are conformal bilinear maps, the $\ast_\lambda$-product defined by \eqref{ast product} satisfies conformal sesquilinearity. Then
$J\natural K$ is called the \textbf{unified product} of $J$ and $K$ associated with $\mathcal U(J, K)$ if it is a Jacobi-Jordan conformal algebra with the $\ast_\lambda$-product given by \eqref{ast product}. In this case, the extending datum $\mathcal U(J, K)$ is called a \textbf{Jacobi-Jordan conformal extending structure} of $J$ through  $K$. We denote by $\mathcal S(J,K)$ the set of all Jacobi-Jordan  conformal extending structures of $J$ through  $K$.

The maps  $\lhd_\lambda$ and $\rhd_\lambda$,  are called the actions of $\mathcal U(J, K)$ and $\omega_\l$ is called the cocycle of $\mathcal U(J, K)$. 
\end{defi}

\begin{rmk}
Note that by Eq. \eqref{ast product}, the following identities  hold in $J \natural J$, for any $a,b \in J$ and $x,y \in K$
\begin{align}
     (a,0)\ast_\l (b,0)=&(a_\l b,0),\qquad \qquad \quad (a,0)\ast_\l (0,y)=(y \rhd_{-\l-\pr}a, y \lhd_{-\l-\pr}a),\\
     (0,x)\ast_\l (b,0)=&(x \rhd_\l b,x \lhd_\l b),\qquad (0,x)\ast_\l (0,y)=(\omega_\l(x,y),x \circ_\l y).
\end{align}
\end{rmk}

The next theorem provides the set of axioms that need to be fulfilled by an extending datum $\mathcal U(J,K)$ such that $J \natural K$ is a unified product.

\begin{thm}\label{unified products}
Let $J$ be a Jacobi-Jordan conformal algebra, $K$ be a $\mathbb F|\pr]$-module and $\mathcal U(J,K)$ be an extending datum of $J$ by $K$. Then $J \natural K$ is a Jacobi-Jordan conformal algebra if and only if the following conditions are satisfied for any $a,b \in J$ and $x,y \in K$
\begin{align*}
  &  {\rm (U_1)}\qquad   \omega_\l(x,y)=\omega_{-\l-\pr}(y,x),\quad x\circ_\l y=y\circ_{-\l-\pr}x,\\
  &  {\rm (U_2)}\qquad     a_\l(x\rhd_{-\mu-\pr}b)+(x\lhd_{-\mu-\pr}b)\rhd_{-\l-\pr} a
 + x\rhd_{-\l-\mu-\pr}(a_\l b) \\
 &\qquad \qquad + b_\mu(x\rhd_{-\l-\pr}a)+(x\lhd_{-\l-\pr}a)\rhd_{-\mu-\pr}b=0,  \\
 &  {\rm (U_3)}\qquad     (x\lhd_{-\mu-\pr}b)\lhd_{-\l-\pr} a
 +x \lhd_{-\l-\mu-\pr}(a_\l b) 
 + 
 (x\lhd_{-\l-\pr}a)\lhd_{-\mu-\pr}b=0.,   \\
&  {\rm (U_4)}\qquad   (a_\l \omega_\mu(x,y)+(x\circ_\mu y)\rhd_{-\l-\pr}a
+(y\rhd_{-\l-\mu-\pr}(x\rhd_{-\l-\pr}a) \\
&\qquad \qquad +\omega_{\l+\mu}(x\lhd_{-\l-\pr}a,y)
+(x\rhd_\mu(y\rhd_{-\l-\pr}a)+\omega_\mu(x,y\lhd_{-\l-\pr}a)=0,  \\
 &  {\rm (U_5)}\qquad    (x\circ_\mu y)\lhd_{-\l-\pr}a) +
y\lhd_{-\l-\mu-\pr}(x\rhd_{-\l-\pr}a)+(x\lhd_{-\l-\pr}a)\circ_{\l+\mu}y) \\
&\qquad \qquad+ x\lhd_\mu(y\rhd_{-\l-\pr}a)+x\circ_\mu(y\lhd_{-\l-\pr}a))=0,  \\
  &  {\rm (U_6)}\qquad      x\rhd_\l \omega_\mu(y,z)+\omega_\l(x, y\circ_\mu z)
 +  y\rhd_{-\l-\mu-\pr}\omega_\l(x,y)+\omega_{\l+\mu}(x\circ_\l y,z) \\
&\hspace{2 cm} + y\rhd_\mu \omega_\l(x,z)+\omega_\mu(y,x \circ_\l z)=0,  \\
  &  {\rm (U_7)}\qquad  x\lhd_\l \omega_\mu(y,z)+x\circ_\l (y\circ_\mu z) 
 +  y\lhd_{-\l-\mu-\pr}\omega_\l(x,y) + (x\circ_\l y)_{\l+\mu}z  \\
& \hspace{2 cm} +
 y\lhd_\mu \omega_\l(x,z)+ y \circ_\mu(x\circ_\l z)=0.   
\end{align*}
\end{thm}
Before going into the proof of the theorem, we make a few remarks on the compatibility in Theorem \ref{unified products}. Aside from the fact that $K$ is not a Jacobi-Jordan conformal algebra, the identities ({\rm $U_3$}) and ({\rm $U_4$})
are exactly the compatibility defining a  matched pair of Jacobi-Jordan conformal  algebras. 
 The compatibility condition ({\rm $U_5$}) is called the twisted module condition for the
action $\lhd_\l$; in the case when $K$ is a Jacobi-Jordan conformal algebra it measures how far $(J,\lhd_\l)$ is from being a left $K$-module.
The condition ({\rm $U_6$}) is called the twisted cocycle condition: if $\lhd_\l$ is the trivial action
and $(K, \circ_\l)$ is a Jacobi-Jordan conformal  algebra then the compatibility condition ({\rm $U_6$}) is exactly the  $2$-cocycle condition for Jacobi-Jordan conformal algebras. Finally, the identity ({\rm $U_7$}) is called the twisted conformal Jacobi condition:
it measures how far $\circ_\l$ is from being a Jacobi-Jordan conformal structure on $K$. If either $\rhd_\l$ or $\omega_\l$ is the
trivial map, then ({\rm $U_7$}) is equivalent to $\circ_\l$ being a Jacobi-Jordan conformal product on $K$.

\begin{proof}
Since $\rhd_\l,\ \lhd_\l,\ \omega_\l$ and $\circ_\l$ are conformal linear maps, then the conformal sesquilinearity of $\ast_\l$ is naturally checked.  Therefore, we need only to show that the commutativity condition Eq. \eqref{commutativity}
and the Jacobi identity Eq. \eqref{conf Jacobi identity} hold for $\ast_\l$ if and only if the identities ${\rm (U_1)}$-${\rm (U_7)}$ are satisfied. \\
First, we can easily prove that $\ast_\l$ is commutative, i.e.  $(a,x)\ast_\l (b,y)=(b,y)\ast_{-\l-\pr} (a,x)$  if and only if 
$\omega_\l(x,y)=\omega_{-\l-\pr}(y,x)$  and $ x\circ_\l y=y\circ_{-\l-\pr}x$, that is ${\rm (U_1)}$ holds.\\
Then $J\natural K$ is a Jacobi-Jordan conformal algebra if and only if 
 the conformal Jacobi-identity Eq. \eqref{conf Jacobi identity} holds for the $\ast_\l$-product, i.e.
 {\small
 \begin{align}
0= &(a,x)\ast_\l((b,y)\ast_\mu (c,z))+((a,x)\ast_\l (b,y))\ast_{\l+\mu}(c,z) 
 + (b,y)\ast_\mu((a,x)\ast_\l (c,z)),\label{F}
\end{align} }
 for all $a$, $b$, $c\in J$ and  $x$, $y$, $z\in K$. 
Since in $J \natural K$, we have $(a,x)=(a,0)+(0,x)$; it follows that \eqref{F} holds if and only if it holds for all generators of $J \natural K$, i.e. the set
$\{(a, 0) | a \in J \} \cup \{(0, x) | x \in K \}$. 
First, we should notice that \eqref{F} holds for the
triple $(a, 0), (b, 0), (c, 0)$ as we have
\begin{align*}
  &(a,0)\ast_\l((b,0)\ast_\mu (c,0))+((a,0)\ast_\l (b,0))\ast_{\l+\mu}(c,0) 
 + (b,0)\ast_\mu((a,0)\ast_\l (c,0)) \\
 =&(a_\l(b_\mu c)+(a_\l b)_{\l+\mu}c+b_\mu(a_\l c),0)=(0,0).
\end{align*}
Next, we prove that \eqref{F} holds for $(a, 0), (b, 0), (0, x)$ if and only if ${\rm (U_2)}$ and ${\rm (U_3)}$ hold. Indeed, we have
\begin{align*}
&(a,0)\ast_\l((b,0)\ast_\mu (0,x))+((a,0)\ast_\l (b,0))\ast_{\l+\mu}(0,x) 
 + (b,0)\ast_\mu((a,0)\ast_\l (0,x)) \\
 =&(a,0)\ast_\l(x \rhd_{-\mu-\pr}b, x \lhd_{-\mu-\pr}b) +(a_\l b,0)\ast_{\l+\mu}(0,x) +(b,0)\ast_\mu (x \rhd_{-\l-\pr}a, x \lhd_{-\l-\pr}a) \\
 =& (a_\l(x\rhd_{-\mu-\pr}b)+(x\lhd_{-\mu-\pr}b)\rhd_{-\l-\pr} a,(x\lhd_{-\mu-\pr}b)\lhd_{-\l-\pr} a) \\
 +& (x\rhd_{-\l-\mu-\pr}(a_\l b),x \lhd_{-\l-\mu-\pr}(a_\l b)) \\
 +& (b_\mu(x\rhd_{-\l-\pr}a)+(x\lhd_{-\l-\pr}a)\rhd_{-\mu-\pr}b, 
 (x\lhd_{-\l-\pr}a)\lhd_{-\mu-\pr}b) . 
\end{align*}
The right hand side varnishes if and only if 
{\small
\begin{align*}
&  a_\l(x\rhd_{-\mu-\pr}b)+(x\lhd_{-\mu-\pr}b)\rhd_{-\l-\pr} a
 + x\rhd_{-\l-\mu-\pr}(a_\l b) \\
 &\qquad + b_\mu(x\rhd_{-\l-\pr}a)+(x\lhd_{-\l-\pr}a)\rhd_{-\mu-\pr}b=0
\end{align*} }
and
\begin{align*}
 &(x\lhd_{-\mu-\pr}b)\lhd_{-\l-\pr} a
 +x \lhd_{-\l-\mu-\pr}(a_\l b) 
 + 
 (x\lhd_{-\l-\pr}a)\lhd_{-\mu-\pr}b=0.
\end{align*}

Thus ${\rm (U_2)}$ and ${\rm (U_3)}$ hold. 
Now, we show that \eqref{F} holds for $(a, 0), (0, x), (0, y)$ if and only if ${\rm (U_4)}$ and ${\rm (U_5)}$ hold. Indeed, we have
{\small
\begin{align*}
&(a,0)\ast_\l((0,x)\ast_\mu (0,y))+((a,0)\ast_\l (0,x))\ast_{\l+\mu}(0,y) 
 + (0,x)\ast_\mu((a,0)\ast_\l (0,y)) \\
= &(a,0)\ast_\l(\omega_\mu(x,y),x\circ_\mu y)+(x\rhd_{-\l-\pr}a,x\lhd_{-\l-\pr}a)\ast_{\l+\mu}(0,y) \\
&\qquad +(0,x)\ast_\mu(y\rhd_{-\l-\pr}a,y\lhd_{-\l-\pr}a) \\
=& (a_\l \omega_\mu(x,y)+(x\circ_\mu y)\rhd_{-\l-\pr}a,(x\circ_\mu y)\lhd_{-\l-\pr}a) \\
+&(y\rhd_{-\l-\mu-\pr}(x\rhd_{-\l-\pr}a)+\omega_{\l+\mu}(x\lhd_{-\l-\pr}a,y),
y\lhd_{-\l-\mu-\pr}(x\rhd_{-\l-\pr}a)+(x\lhd_{-\l-\pr}a)\circ_{\l+\mu}y)\\
+&(x\rhd_\mu(y\rhd_{-\l-\pr}a)+\omega_\mu(x,y\lhd_{-\l-\pr}a),
x\lhd_\mu(y\rhd_{-\l-\pr}a)+x\circ_\mu(y\lhd_{-\l-\pr}a)).
\end{align*}}
Then $(a,0)\ast_\l((0,x)\ast_\mu (0,y))+((a,0)\ast_\l (0,x))\ast_{\l+\mu}(0,y) 
 + (0,x)\ast_\mu((a,0)\ast_\l (0,y))=0$ if and only if
 \begin{align*}
   &  (a_\l \omega_\mu(x,y)+(x\circ_\mu y)\rhd_{-\l-\pr}a
+(y\rhd_{-\l-\mu-\pr}(x\rhd_{-\l-\pr}a) \\
+&\omega_{\l+\mu}(x\lhd_{-\l-\pr}a,y)
+(x\rhd_\mu(y\rhd_{-\l-\pr}a)+\omega_\mu(x,y\lhd_{-\l-\pr}a)=0
 \end{align*}
 and
  \begin{align*}
 &    (x\circ_\mu y)\lhd_{-\l-\pr}a) +
y\lhd_{-\l-\mu-\pr}(x\rhd_{-\l-\pr}a)+(x\lhd_{-\l-\pr}a)\circ_{\l+\mu}y) \\
+& x\lhd_\mu(y\rhd_{-\l-\pr}a)+x\circ_\mu(y\lhd_{-\l-\pr}a))=0.
 \end{align*}
Thus ${\rm (U_4)}$ and ${\rm (U_5)}$ hold.
Finally, we prove that \eqref{F} holds for $(0,x), (0,y), (0,z)$ if and only if ${\rm (U_6)}$ and ${\rm (U_7)}$ hold. Indeed, we have
\begin{align*}
     &(0,x)\ast_\l((0,y)\ast_\mu (0,z))+((0,x)\ast_\l (0,y))\ast_{\l+\mu}(0,z) 
 + (0,y)\ast_\mu((0,x)\ast_\l (0,z)) \\
 =& (0,x)\ast_\l(\omega_\mu(y,z),y\circ_\mu z)+ (\omega_\l(x,y),x\circ_\l y)\ast_{\l+\mu}(0,z) 
 +(0,y)\ast_\mu(\omega_\l(x,z),x\circ_\l z) \\
 =& (x\rhd_\l \omega_\mu(y,z)+\omega_\l(x, y\circ_\mu z),
 x\lhd_\l \omega_\mu(y,z)+x\circ_\l (y\circ_\mu z)) \\
 + & (y\rhd_{-\l-\mu-\pr}\omega_\l(x,y)+\omega_{\l+\mu}(x\circ_\l y,z), y\lhd_{-\l-\mu-\pr}\omega_\l(x,y) + (x\circ_\l y)_{\l+\mu}z) \\
 +& (y\rhd_\mu \omega_\l(x,z)+\omega_\mu(y,x \circ_\l z),
 y\lhd_\mu \omega_\l(x,z)+ y \circ_\mu(x\circ_\l z)).
\end{align*}
Therefore 
$(0,x)\ast_\l((0,y)\ast_\mu (0,z))+((0,x)\ast_\l (0,y))\ast_{\l+\mu}(0,z) 
 + (0,y)\ast_\mu((0,x)\ast_\l (0,z))=0$ if and only if 
 \begin{align*}
  &   x\rhd_\l \omega_\mu(y,z)+\omega_\l(x, y\circ_\mu z)
 +  y\rhd_{-\l-\mu-\pr}\omega_\l(x,y)+\omega_{\l+\mu}(x\circ_\l y,z) \\
&\hspace{2 cm} + y\rhd_\mu \omega_\l(x,z)+\omega_\mu(y,x \circ_\l z)=0
\end{align*} 
and 
\begin{align*}
 & x\lhd_\l \omega_\mu(y,z)+x\circ_\l (y\circ_\mu z) 
 +  y\lhd_{-\l-\mu-\pr}\omega_\l(x,y) + (x\circ_\l y)_{\l+\mu}z  \\
& \hspace{2 cm} +
 y\lhd_\mu \omega_\l(x,z)+ y \circ_\mu(x\circ_\l z)=0.
 \end{align*} 
Hence the proof is finished. 
\end{proof}
In the sequel, we use the following convention: if one of the maps $\lhd_\l$, $\rhd_\l$, $\omega_\l$, or $\circ_\l$ of
an extending datum is trivial, we will omit it from 
$\mathcal U(J,K) = \{\lhd_\l, \rhd_\l, \omega_\l, \circ_\l \}$. 

From now on, in light of Theorem \ref{unified products}, a Jacobi-Jordan conformal extending structure of
$J$ through $K$ will be viewed as a system
$\mathcal U(J,K) =  \{\lhd_\l, \rhd_\l, \omega_\l, \circ_\l \}$ satisfying the compatibility conditions
$({\rm U}_1)$-$({\rm U}_7)$. 

\begin{ex}
Let $\mathcal U(J,K) =  \{\lhd_\l, \rhd_\l, \omega_\l, \circ_\l \}$ be an extending datum of a Jacobi-Jordan conformal algebra $J$ through a $\mathbb F [\pr]$-module $K$, where $\rhd_\l, \omega_\l$ and $\circ_\l$ are trivial conformal bilinear maps. 
Then the corresponding extending datum denoted by $\mathcal U(J,K)=\{\lhd_\l \}$ is a Jacobi-Jordan conformal extending structure of $J$ by $K$ if and only if $K$ is a right $K$-module under $\rhd_\l$. Therefore, the associated unified product is just the semi-direct product  of $J$ and $K$. 
\end{ex}
Given an extending structure $\mathcal U(J,K )$. It is obvious that $J$ can be seen as a Jacobi-Jordan conformal 
subalgebra of $J \natural K$. Conversely, we will prove that any Jacobi-Jordan conformal  algebra structure on a
vector space E containing $J$ as a subalgebra is isomorphic to a unified product.
\begin{thm}\label{reciproque}
    Let $J$ be a Jacobi-Jordan conformal algebra and $K$ an $\mathbb F[\pr]$-module. Set $E=J \oplus K$, where the direct sum is the sum of $\mathbb F[\pr]$-modules. Suppose that $E$ has a Jacobi-Jordan  conformal algebra structure $\ast_\l$ such that $J$ is a subalgebra. Then 
there exists a Jacobi-Jordan conformal   extending structure  $\mathcal U(J,K) =  \{\lhd_\l, \rhd_\l, \omega_\l, \circ_\l \}$ of $J$ by $K$ and an isomorphism of Jacobi-Jordan conformal algebra $E \cong J \natural K$ which stabilizes $J$ and co-stabilizes $K$. 
\end{thm}
\begin{proof}
  Since $E=J \oplus K$, there exists a natural  $\mathbb F[\pr]$-homomorphism $p: E \to J$ such that $p(a)=a$, for any $a \in J$.  Then, we can define an extending datum  $\mathcal U(J,K) =  \{\lhd_\l, \rhd_\l, \omega_\l, \circ_\l \}$ of $J$ by $K$ as follows ($a \in J$ and $x,y \in K$)
  \begin{align}
\rhd_\l : K \times J \to
J[\l], \qquad x \rhd_\l  a =& p (x \ast_\l a ) \label{bala1}\\
\lhd_\l : K \times J \to K[\l],
\qquad x \lhd_\l a =& x \ast_\l a - p (x \ast_\l  a) \label{bala2} \\
\omega_\l : K \times K \to J[\l], \qquad \omega_\l(x, y) =&
p (x \ast_\l y)  \label{bala3} \\
\circ_\l: K \times K \to K[\l], \qquad x\circ_\l \, y =&
x\ast_\l  y  - p (x \ast_\l  y). \label{bala4}
\end{align}
  Obviously, $\lhd_\l$, $\rhd_\l$, $\omega_\l$ and  $\circ_\l$ are four conformal bilinear maps. Next, we shall prove that  $\mathcal U(J,K) =  \{\lhd_\l, \rhd_\l, \omega_\l, \circ_\l \}$ of $J$ is a Jacobi-Jordan conformal extending structure and $E  \cong J \natural K$ as Jacobi-Jordan conformal algebras. 
  Set $\psi: J \times K \to E$ by $\psi(a,x)=a+x$. It is easy to see that $\psi $ is an $\mathbb F[\pr]$-module isomorphism. Its inverse is $\psi^{-1}: E \to J \times K$ defined by $\psi^{-1}(e)=(p(e),e-p(e))$. Therefore, if $\psi$ is also an isomorphism of Jacobi-Jordan conformal algebra, there exists a  unique $\l$-product  $\square_\l$ on $J \times K$ given by 
  \begin{align}
      (a,x)\square_\l (b,y)=\psi^{-1}(\psi(a,x)\ast_\l \psi(b,y)), \label{square product}
  \end{align}
  for any $a,b \in J$ and $x,y \in K$.  Hence, for finishing the proof, we only need to prove
that the $\l$-product defined by \eqref{square product} is just the one given by  
\eqref{ast product} associated with
the above extending system $\mathcal U(J,K) =  \{\lhd_\l, \rhd_\l, \omega_\l, \circ_\l \}$.  We check it as follows
\begin{align*}
     &(a,x)\square_\l (b,y)  \\
       =& \psi^{-1}(\psi(a,x)\ast_\l \psi(b,y))    \\
    =& \psi^{-1}((a+x)\ast_\l (b+y)) \\
    =& \psi^{-1}(a \ast_\l b+ a \ast_\l y +x \ast_\l b+ x \ast_\l y ) \\
    =& (p(a \ast_\l b)+ p(a \ast_\l y) +p(x \ast_\l b)+ p(x \ast_\l y), \\
   & a \ast_\l b+ a \ast_\l y +x \ast_\l b+ x \ast_\l y-p(a \ast_\l b)-p( a \ast_\l y) -p(x \ast_\l b)- p(x \ast_\l y)) \\
    =& (a_\l b +p(y \ast_{-\l-\pr} a) +p(x \ast_\l b)+ p(x \ast_\l y), \\
  &  .y \ast_{-\l-\pr} a-p(y \ast_{-\l-\pr} a) +x \ast_\l b-p(x \ast_\l b)+x \ast_\l y-p(x \ast_\l y)) \\
    =& (a_\l b+ y \rhd_{-\l-\pr}a + x \rhd_\l b+ \omega_\l (x,y), 
     y \lhd_{-\l-\pr} a + x \lhd_\l b+ x \circ_\l y ) \\
     =& (a,x) \ast_\l (b,y). 
\end{align*}
Moreover, it is easy to see that $\psi$ stabilizes $J$ and co-stabilizes $K$.
Then, the proof is finished. 
\end{proof}

\begin{defi}
Let $J$ be a Jacobi-Jordan conformal algebra and $K$ an $\mathbb F[\pr]$-module. Two Jacobi-Jordan conformal algebra extending structures of $J$ by $K$, $\mathcal U(J,K) =  \{\lhd_\l, \rhd_\l, \omega_\l, \circ_\l \}$ and $\mathcal U'(J,K) =  \{\lhd'_\l, \rhd'_\l, \omega'_\l, \circ'_\l \}$ are called {\it equivalent} and we denote this by $\mathcal U(J,K) \equiv \mathcal U'(J,K)$, if there exists a pair of $\mathbb F [\pr]$-module homomorphisms $(r,s)$, where $r: K \to J$ and $s \in {\rm Aut}_{\mathbb F[\pr]} (K)$ such that $\{\lhd'_\l, \rhd'_\l, \omega'_\l, \circ'_\l \}$ is implemented from  $\{\lhd_\l, \rhd_\l, \omega_\l, \circ_\l \}$ using $(r,s)$ via:
{\small
\begin{align}
&    x \rhd_{-\l-\pr} a+ r(x \lhd_{-\l-\pr} a)=a_\l r(x)+s(x) \rhd'_{-\l-\pr} a  ,  \label{rs1} \\
  &   s(x \lhd_{-\l-\pr} a)=  s(x) \lhd'_{-\l-\pr} a ,  \label{rs2} \\
 &   \omega_\l(x,y)+ r(x \circ_\l y)=  r(x)_\lambda r(y)+s(x)\rhd'_\lambda r(y)+s(y) \rhd'_{-\l-\pr}r(x) +\omega'_\lambda(s(x),s(y)) ,   \label{rs3}\\
   &   s(x \circ_\l y)=  s(x)\lhd'_\lambda r(y)+s(y) \lhd'_{-\l-\pr}r(x) +s(x)\circ'_\lambda s(y) , \label{rs4}
\end{align}  }
for any $a \in J$ and $x,y \in K$. 
In particular, if $s=id$, $\mathcal U(J,K)$ and  $\mathcal U'(J,K)$  are called {\it cohomologous} and we denote it by $\mathcal U(J,K) \approx
 \mathcal U'(J,K)$. 
\end{defi}

Theorem \ref{reciproque} shows that the classification of all Jacobi-Jordan conformal  algebra structures on $E$ that
contain $J$ as a subalgebra comes down to the classification of the unified products $J \natural K$
for a given complement $K$ of $J$ in $E$. In order to describe the classifying sets $Jextd (E,J)$, we need the following:

\begin{lem}
Suppose that $\mathcal U(J,K) =  \{\lhd_\l, \rhd_\l, \omega_\l, \circ_\l \}$ and $\mathcal U'(J,K) =  \{\lhd'_\l, \rhd'_\l, \omega'_\l, \circ'_\l \}$ are two Jacobi-Jordan conformal extending structures of $J$ by $K$. Let $J \natural K$ and $J \natural' K$ be the corresponding unified products. Then $J \natural K\equiv  J \natural' K$ if and only if $\mathcal U(J,K) \equiv \mathcal U'(J,K)$ and  $J \natural K \approx  J \natural' K$ if and only if $\mathcal U(J,K) \approx \mathcal U'(J,K)$.
\end{lem}

\begin{proof}
Let $\phi: J \natural K \to J \natural' K$ be an isomorphism of Jacobi-Jordan conformal algebras which stabilizes $J$. According to that $\phi$ stabilizes $J$, $\phi(a,0)=(a,0)$. Moreover, we can set $\phi(0,x)=(r(x),s(x))$, where $r: K \to J$ and $s: K \to K$ are two linear maps. Therefore, we get $\phi(a,x)=(a+r(x),s(x))$.
Furthermore, it is easy to see that $\phi$ is an $\mathbb F[\pr]$-module homomorphism if and only if $r$ and $s$ are two $\mathbb F[\pr]$-module homomorphisms. 
Then, we shall prove that $\phi$ is a Jacobi-Jordan conformal algebra homomorphism, i.e.
\begin{align}
 \phi((a,x)\ast_\l (b,y))=\phi(a,x) \ast'_\l \phi(b,y),    \label{fi morph}
\end{align}
for any $a,b \in J$ and $x,y \in K$ if and only if Eqs. \eqref{rs1}-\eqref{rs4} hold. It is enough to check that \eqref{fi morph} holds for all generators of $J \natural K$. Obviously, \eqref{fi morph}  holds for the pair $(a,0)$, $(b,0)$. Then, we consider \eqref{fi morph}  for the pair $(a,0)$, $(0,x)$. According to 
\begin{align*}
     \phi((a,0)\ast_\l (0,x))= &\phi(x \rhd_{-\l-\pr} a, x \lhd_{-\l-\pr} a) \\
    =& (x \rhd_{-\l-\pr} a+ r(x \lhd_{-\l-\pr} a),s(x \lhd_{-\l-\pr} a) )
\end{align*}
 and 
 \begin{align*}
     \phi(a,0) \ast'_\l \phi(0,x)=& (a,0) \ast'_\l (r(x),s(x)) \\
     =& (a_\l r(x)+s(x) \rhd'_{-\l-\pr} a, s(x) \lhd'_{-\l-\pr} a),
 \end{align*}
 we get that \eqref{fi morph} holds for the pair $(a,0)$, $(0,x)$ if and only if 
 \eqref{rs1} and \eqref{rs2} hold.  Next, we consider  \eqref{fi morph}  for the pair $(0,x)$, $(0,y)$. We have 
 \begin{align*}
    \phi((0,x)\ast_\l (0,y))=& \phi(\omega_\l(x,y),x \circ_\l y) \\
    =& (\omega_\l(x,y)+ r(x \circ_\l y), s(x \circ_\l y) )
 \end{align*}
and 
 \begin{align*}
      \phi(0,x) \ast'_\l \phi(0,y) =& (r(x),s(x))\ast'_\l (r(y),s(y)) \\
      =& (r(x)_\lambda r(y)+s(x)\rhd'_\lambda r(y)+s(y) \rhd'_{-\l-\pr}r(x) +\omega'_\lambda(s(x),s(y)) ,\\
      &s(x)\lhd'_\lambda r(y)+s(y) \lhd'_{-\l-\pr}r(x) +s(x)\circ'_\lambda s(y) ),
 \end{align*}
 which means that \eqref{fi morph} holds if and only if \eqref{rs3} and \eqref{rs4} hold. 
\end{proof}

By putting together all the results proved in this section we obtain the following theoretical answer to the extending structures problem for Jacobi-Jordan conformal  algebras
\begin{thm}
Let $J$ be a Jacobi-Jordan conformal algebra and $K$ be an $\mathbb F[\pr]$-module. Set $E=J \oplus K$, where the direct sum is the sum of $\mathbb F[\pr]$-modules. Then, we get:    
\begin{enumerate}
    \item  Denote $\mathcal H^2_J(K,J):=\mathcal S(J,K)/\equiv$. Then the map 
   $$\mathcal H^2_J(K,J) \to Extd(E,J),\quad  \overline{\mathcal U(J,K)} \mapsto (J \natural K,\ast_\l)$$  
   is bijective, where $\overline{\mathcal U(J,K)}$ is the equivalence class of  $\mathcal U(J,K)$ with respect to  $\equiv$. 
    \item  Denote $\mathcal H^2(K,J):=\mathcal S(J,K)/\approx$. Then the map 
   $$\mathcal H^2(K,J) \to Extd(E,J),\quad  \overline{\overline{\mathcal U(J,K)}} \mapsto (J \natural K,\ast_\l)$$  
   is bijective, where $\overline{\overline{\mathcal U(J,K)}}$ is the equivalence class of  $\mathcal U(J,K)$ with respect to $\approx$. 
\end{enumerate}
\end{thm}

\subsection{Applications to special cases of unified products}
In this section, we introduce some interesting products of Jacobi-Jordan  conformal algebras such as twisted products, crossed products and bicrossed products which are all special cases of unified products. These products will be useful for studying the structure theory of Jacobi-Jordan  conformal algebras.
\subsubsection{Twisted products of Jacobi-Jordan conformal algebras}
Let $\mathcal U(J,K) =  \{\lhd_\l, \rhd_\l, \omega_\l, \circ_\l \}$ be an extending datum of a Jacobi-Jordan conformal algebra $J$ through an $\mathbb F [\pr]$-module $K$, where $\rhd_\l$ and $\lhd_\l$ are trivial conformal bilinear maps.
Then by Theorem \ref{unified products},  $\mathcal U(J,K) =  \{ \omega_\l, \circ_\l \}$ is a Jacobi-Jordan conformal extending structure of $J$ by $K$ if and only if $(K,\circ_\l)$ is a Jacobi-Jordan conformal algebra and $\omega_\l: K \times K \to J[\l]$ satisfies 
\begin{align}
    & \omega_\l(x,y)=\omega_{-\l-\pr}(y,x)  ,\\
    & \omega_\l(x, y\circ_\mu z)
 + \omega_{\l+\mu}(x\circ_\l y,z) 
 +\omega_\mu(y,x \circ_\l z)=0,
\end{align}
for any $a \in J$ and $x,y,z \in K$. The associated unified product $J\natural K $
is called the {\it twisted  product} of $J$ and $K$. The product on $J \natural K$ is given, for any $a,b \in J$ and $x,y \in K$ by
\begin{align}
(a,x)\ast_\lambda (b,y)
=(a_\lambda b +\omega_\lambda(x,y),x\circ_\lambda y).  
\label{ast twisted product}
\end{align}

\subsubsection{Crossed products of Jacobi-Jordan conformal algebras}
Let $\mathcal U(J,K) =  \{\lhd_\l, \rhd_\l, \omega_\l, \circ_\l \}$ be an extending datum of a Jacobi-Jordan conformal algebra $J$ through an $\mathbb F [\pr]$-module $K$, where $\lhd_\l$ is the trivial conformal bilinear map. Then $\mathcal U(J,K)$ is a Jacobi-Jordan conformal extending structure of $J$ through $K$ if and only if 
$(K,\circ_\l)$ is a Jacobi-Jordan conformal algebra and the following conditions are satisfied:
\begin{align}
  &      \omega_\l(x,y)=\omega_{-\l-\pr}(y,x),\\
  &      a_\l(x\rhd_{-\mu-\pr}b)
 + x\rhd_{-\l-\mu-\pr}(a_\l b)  + b_\mu(x\rhd_{-\l-\pr}a)=0,  \\
&    a_\l \omega_\mu(x,y)+(x\circ_\mu y)\rhd_{-\l-\pr}a
+y\rhd_{-\l-\mu-\pr}(x\rhd_{-\l-\pr}a) 
+x\rhd_\mu(y\rhd_{-\l-\pr}a)=0,  \\
  &        x\rhd_\l \omega_\mu(y,z)+\omega_\l(x, y\circ_\mu z)
 +  y\rhd_{-\l-\mu-\pr}\omega_\l(x,y)+\omega_{\l+\mu}(x\circ_\l y,z) \\
&\hspace{2 cm} + y\rhd_\mu \omega_\l(x,z)+\omega_\mu(y,x \circ_\l z)=0
\end{align}
for any $a,b \in J$ and $x,y,z \in K$. The associated unified product is  denoted by $J\natural^\omega_\rhd K $
and  called the crossed product of $J$ through $K$. The $\lambda$-product  on $J\natural^\omega_\rhd K $ is given by, for any $a,b\in J$ and $x,y\in K,$
$$(a,x)\ast_\lambda (b,y)
=(a_\lambda b+x\rhd_\lambda b+y \rhd_{-\l-\pr}a +\omega_\lambda(x,y), x\circ_\lambda y). $$
It is obvious that $J$ is an ideal of $J\natural^\omega_\rhd K $ and one  can directly obtain by Theorem \ref{unified products} the following result:
\begin{pro}
Let $J$ be a Jacobi-Jordan  conformal algebra and $K$ a $\mathbb C[\pr]$-module. Set $E = J\oplus K$ where the direct sum is the sum of $\mathbb C[\pr]$-modules. Assume that $E$ has a Jacobi-Jordan  conformal
algebra structure $\ast_\lambda$ such that $J$ is an ideal of $E$. Then, $E$ is isomorphic to a crossed
product $J\natural^\omega_\rhd K $ of Jacobi-Jordan conformal algebras.
\end{pro}

\end{document}